\documentclass[12pt]{amsart}
\usepackage{amsmath}
\usepackage{amsfonts}
\usepackage{amssymb}
\usepackage{amscd}
\usepackage{graphicx}
\RequirePackage[dvipsnames,usenames]{color}
\usepackage{soul,xcolor}
\setstcolor{red}
\usepackage{stmaryrd}
\SetSymbolFont{stmry}{bold}{U}{stmry}{m}{n}
\usepackage{mathtools}
\usepackage{booktabs}
\usepackage{multirow}\usepackage{subcaption}

\usepackage{mathdots}
\newtagform{tiny}{\tiny(}{)}

\usepackage{mathtools}
\usepackage{hyperref}
\usepackage[margin=1.25in]{geometry}

\usepackage{amsthm}
\usepackage{comment}
\usepackage[all,cmtip]{xy}
\usepackage{tikz-cd}
\usetikzlibrary{cd}

\usepackage[all]{xy}

\usepackage{etoolbox}
\apptocmd{\sloppy}{\hbadness 10000\relax}{}{}

\usepackage{enumitem}

\title{On Steenbrink vanishing for rational singularities in positive characteristic}
\author{Tatsuro Kawakami}
\address{Graduate School of Mathematical Sciences, University of Tokyo, 3-8-1 Komaba,
Meguro-ku, Tokyo 153-8914, Japan}
\email{tatsurokawakami0@gmail.com}

\def\phi{\varphi}
\def\epsilon{\varepsilon}

\def\onto{\relbar\joinrel\twoheadrightarrow}

\def\log{\operatorname{log}}

\def\Spec{\operatorname{Spec}}

\def\Supp{\operatorname{Supp}}

\def\Ker{\operatorname{ker}}

\def\m{{\mathfrak m}}

\newcommand{\Q}{\mathbb{Q}}

\newcommand{\Z}{\mathbb{Z}}

\newcommand{\sO}{\mathcal{O}}
\newcommand{\sHom}{\mathcal{H}\! \mathit{om}}

\newsavebox{\pullback}
\sbox\pullback{%
\begin{tikzpicture}%
\draw (0,0) -- (1ex,0ex);%
\draw (1ex,0ex) -- (1ex,1ex);%
\end{tikzpicture}}

\newsavebox{\pullbackdl}
\sbox\pullbackdl{%
\begin{tikzpicture}%
\draw (-1ex,0ex) -- (0ex,0ex);%
\draw (0ex,-1ex) -- (0ex,0ex);%
\end{tikzpicture}}

\newsavebox{\pushoutdr}
\sbox\pushoutdr{%
\begin{tikzpicture}%
\draw (-1ex,-1ex) -- (-1ex,0ex);%
\draw (-1ex,0ex) -- (0ex,0ex);%
\end{tikzpicture}}

\theoremstyle{plain}
\newtheorem{thm}{Theorem}[section] 

\newtheorem{prop}[thm]{Proposition}

\newtheorem{lem}[thm]{Lemma}
\theoremstyle{definition} 
\newtheorem{defn}[thm]{Definition}

\newtheorem{eg}[thm]{Example} 

\theoremstyle{remark}
\newtheorem{rem}[thm]{Remark}

\newtheorem{defn and notation}[thm]{Definition and Notation}

\theoremstyle{plain}
\newtheorem{theo}{Theorem}

\numberwithin{equation}{thm}

\begin{document}
\tolerance = 9999

\begin{abstract}
We show a special case of Steenbrink vanishing for rational singularities in positive characteristic.
As a consequence, we prove that strongly $F$-regular threefolds and klt threefolds in characteristic $p>41$ satisfy Steenbrink vanishing.
\end{abstract}

\maketitle
\markboth{Tatsuro Kawakami}{STEENBRINK VANISHING IN POSITIVE CHARACTERISTIC}

\section{Introduction}

The Steenbrink vanishing discussed in this paper is a local property of algebraic varieties, defined as follows:
\begin{defn}[Steenbrink vanishing]\label{def:Steenbrink vanishing}
  Let $X$ be a normal variety over a perfect field, and let $\pi\colon Y\to X$ be a log resolution with reduced $\pi$-exceptional divisor $E$.
  Then 
  \[
  R^j\pi_{*}\Omega^i_Y(\log E)(-E)=0
  \]
  for all integers $i,j$ satisfying $i+j>\dim X$.
\end{defn}

Here, $R^j\pi_{*}\Omega^i_Y(\log E)(-E)$ denotes $R^j\pi_{*}\left(\Omega^i_Y(\log E)\otimes\sO_Y(-E) \right)$.
Steenbrink vanishing includes Grauert--Riemenschneider vanishing as a special case when $i=\dim X$.
While Grauert–Riemenschneider vanishing can be viewed as a local version of Kodaira vanishing, Steenbrink vanishing can be regarded as a local analogue of Akizuki–Nakano vanishing.

In characteristic zero, Steenbrink vanishing was proven by Steenbrink himself \cite[Theorem 2(b)]{Steenbrink} without any assumptions on the singularities (see also \cite[Theorem 3.2]{Kovacs(Steenbrink)}).
Moreover, when $X$ is Du Bois, such vanishing is known to hold for a wider range of $i$ and $j$ (see \cite[Theorem 14.1]{GKKP} and \cite[Theorem 1.1]{Kovacs(Steenbrink)}).

In this paper, we study Steenbrink vanishing in positive characteristic.
Since the proof in characteristic zero relies on mixed Hodge theory, similar arguments do not apply in positive characteristic.
Indeed, for any positive integer $d\geq 3$ and any prime number $p$, there exists a $d$-dimensional isolated singularity $X$ over an algebraically closed field of characteristic $p$ that fails to satisfy Steenbrink vanishing.
Such a counterexample is constructed by taking the affine cone over a smooth projective variety that violates Akizuki--Nakano vanishing (see Example \ref{eg:counterex} for details).

On the other hand, for the development of birational geometry in positive characteristic, it is both interesting and important to identify natural and reasonably mild classes of singularities for which Steenbrink vanishing holds.

Thus, this paper investigates classes of singularities where Steenbrink vanishing is valid.
In particular, we show that Steenbrink vanishing partially holds for rational singularities (Definition \ref{def:rational singularities}); moreover, for three-dimensional $F$-injective rational singularities, Steenbrink vanishing holds in all ranges.

\begin{theo}\label{Introthm:F-inj}
    Let $X$ be a normal variety with rational singularities over a perfect field of positive characteristic with $d\coloneqq \dim X\geq 3$. When $d=3$, we assume that $X$ is $F$-injective.
    Then for every log resolution $\pi\colon Y\to X$ whose reduced $\pi$-exceptional divisor $E$ supports a $\pi$-ample divisor, we have
    \[
    R^{d-1}\pi_{*}\Omega^{d-1}_Y(\log E)(-E)=0.
    \]
\end{theo}
\begin{rem}
    \begin{enumerate}
       \item The condition that $E$ supports a $\pi$-ample divisor is automatically satisfied when $X$ is $\Q$-factorial. In general, if log resolutions exist in dimension $d$, then 
        we can find such a log resolution \cite[Theorem 1.1]{Kollar-Witaszek}.
        \item In the case $d=2$, we refer to \cite{Kaw6}. Note that when $d=2$, the above vanishing falls outside the scope of Steenbrink vanishing, since $(d-1)+(d-1)=d$. 
        \item Without any assumptions on the singularities, the theorem does not hold (see Example \ref{eg:counterex}).
    \end{enumerate}
\end{rem}

To the best of our knowledge, Theorem \ref{Introthm:F-inj} is new even when $X$ is smooth. Note that smooth varieties have rational singularities in all characteristics \cite{CR11,CR15}.

\subsection{Strongly F-regular threefolds}
Combining Theorem \ref{Introthm:F-inj} with the fact that strongly $F$-regular threefolds have rational singularities \cite[Theorem 3.5]{Baudin-Kawakami-Rysler}, we obtain the following result:

\begin{theo}\label{Introthm:SFR}
    Let $X$ be a strongly $F$-regular threefold over a perfect field of positive characteristic.
    Then 
    for every log resolution $\pi\colon Y\to X$ whose reduced $\pi$-exceptional divisor $E$ supports a $\pi$-ample divisor,
    we have
    \[
    R^j\pi_{*}\Omega^{i}_Y(\log E)(-E)=0
    \]
    for all integers $i,j$ satisfying $i+j>3$.
    In particular, a $\Q$-factorial strongly $F$-regular threefold satisfies Steenbrink vanishing.
\end{theo}
The essential case in Theorem \ref{Introthm:SFR} is $R^2\pi_{*}\Omega^{2}_Y(\log E)(-E)=0$.
Indeed, it is shown in \cite[Corollary 1.4]{Baudin-Kawakami-Rysler} that a milder assumption suffices to ensure the vanishing of higher direct images of the dualizing sheaf: if $X$ is a three-dimensional Noetherian excellent Cohen–Macaulay scheme of klt type, then $R^j\pi_{*}\omega_Y=0$ holds for all $j>0$.

\subsection{Klt threefolds in large characteristic}

It is known that three-dimensional klt singularities are rational when the characteristic $p>5$ \cite{ABL,Bernasconi-Kollar}.
Moreover, $\mathbb{Q}$-factorial three-dimensional klt singularities over a perfect field of characteristic $p>41$ are quasi-$F$-pure, a milder notion than $F$-purity \cite[Theorem A]{KTTWYY2}.

In Theorem \ref{Introthm:F-inj}, the assumption of $F$-injectivity can be weakened to the assumption of being quasi-$F$-injective (Definition \ref{def:quasi-F-injective}), a broader class related to quasi-$F$-purity.
As a consequence, we obtain Steenbrink vanishing for $\mathbb{Q}$-factorial three-dimensional klt singularities in characteristic $p>41$.

\begin{theo}\label{Introthm:klt}
    Let $X$ be a $\Q$-factorial klt threefold over a perfect field of characteristic $p>41$.
    Then $X$ satisfies Steenbrink vanishing, i.e., 
    for every log resolution $\pi\colon Y\to X$ with reduced $\pi$-exceptional divisor $E$,
    we have
    \[
    R^j\pi_{*}\Omega^{i}_Y(\log E)(-E)=0
    \]
    for all integers $i,j$ satisfying $i+j>3$.
\end{theo}
Again, we remark that the essential case in the theorem is when $(i,j)=(2,2)$; the remaining cases are already known to hold when $p>5$ \cite{ABL,Bernasconi-Kollar}.

\subsection{Strategy of Proof of Theorem \ref{Introthm:F-inj}}

The proofs of the main theorems consist of two steps.
The first step is to establish the following Steenbrink-type vanishing, up to the action of the inverse Cartier operator.

\begin{theo}[Theorem \ref{thm:steenbring vanishing up to nilpotent}]\label{Intromain:up to C} 
Let $X$ be a normal variety over a perfect field of positive characteristic with $d\coloneqq\dim X \geq 2$, and let $\pi\colon Y\to X$ be a log resolution whose reduced $\pi$-exceptional divisor $E$ supports a $\pi$-ample divisor.
Then $R^{d-1}\pi_{*}\Omega^i_Y(\log E)(-E)$ is annihilated by the action of the inverse iterated Cartier operator for all $i\geq 0$. \end{theo}

We refer to Theorem \ref{thm:steenbring vanishing up to nilpotent} for the precise description of the action of the inverse iterated Cartier operator.
Note that while the degree of the higher direct image is $d-1$, there is no restriction on the degree $i$ of $\Omega^i_Y(\log E)(-E)$.

As the next step, we consider conditions under which the action of the inverse iterated Cartier operator becomes injective.

\begin{theo}[Propositions \ref{prop:higherdim} and \ref{prop:threedim}]\label{introthm:injective action}
    Let $X$ be a normal variety over a perfect field of positive characteristic, and let $\pi\colon Y\to X$ be a log resolution with reduced $\pi$-exceptional divisor $E$.
    Suppose that $X$ has rational singularities.
    Then the action of the inverse iterated Cartier operator on
    $R^{j}\pi_{*}\Omega^{d-1}_Y(\log E)(-E)$ 
    is injective for all $j\geq 3$.
    Moreover, $X$ is quasi-$F$-injective in addition, then the action of the inverse iterated Cartier operator on
    $R^{2}\pi_{*}\Omega^{d-1}_Y(\log E)(-E)$
    is injective.
\end{theo}
Note that while in Theorem \ref{Intromain:up to C} the degree of the higher direct image is fixed at $d-1$, here in Theorem \ref{introthm:injective action} the degree $j$ of the higher direct image can be any integer bigger than one.
Combining Theorem \ref{Intromain:up to C} and Theorem \ref{introthm:injective action}, we obtain Theorem \ref{Introthm:F-inj}.

\section{Preliminaries}

\subsection{Notation and terminology}
Throughout the paper, we work over a fixed perfect field $k$ of characteristic $p > 0$ unless stated otherwise.
A \textit{variety} is an integral separated scheme of finite type over $k$. 
Given a Noetherian normal integral scheme $X$, a projective birational morphism $\pi \colon Y \to X$ is called \emph{a resolution of $X$} if $Y$ is regular.
If the reduced $\pi$-exceptional divisor is a simple normal crossing divisor (snc for short) in addition, then
$\pi \colon Y \to X$ is called \emph{a log resolution of $X$}.

\subsection{Quasi-\texorpdfstring{$F$}{F}-injective}

Let $X$ be a normal $F$-finite Noetherian scheme.
For a sheaf of Witt vectors, we have the following maps: 
\begin{alignat*}{4}
&(\textrm{Frobenius}) &&F \colon W_n\sO_X \to W_n\sO_X, \quad &&F(a_0, a_1, \ldots, a_{n-1}) &&\coloneqq (a_0^p,a_1^p, \ldots, a_{n-1}^p),\\ 
&(\textrm{Verschiebung}) \quad  && V \colon W_n\sO_X \to W_{n+1}\sO_X, \quad  &&V(a_0, a_1, \ldots, a_{n-1}) &&\coloneqq (0, a_0, a_1,\ldots,  a_{n-1}),\\
&(\textrm{Restriction}) && R \colon W_{n+1}\sO_X \to W_{n}\sO_X, \quad  &&R(a_0, a_1, \ldots, a_{n}) &&\coloneqq (a_0, a_1, \ldots,  a_{n-1}),
\end{alignat*}
where $F$ and $R$ are ring homomorphisms and $V$ is an additive homomorphism.
We denote the composition $W_n\sO_X \to \sO_X$ of restrictions by $R_{n-1}$.
We then have the following short exact sequence:
\begin{equation}\label{eq:WWO}
    0 \to F_*W_{n-1}\sO_X \xrightarrow{V} W_n \sO_X \xrightarrow{R_{n-1}} \sO_X \to 0
\end{equation}
of $W_n\sO_X$-modules.
We define a $W_n\sO_X$-module $\mathcal{B}_{n,X}$ by \[
\mathcal{B}_{n,X}\coloneqq\mathrm{coker}(F\colon W_n\sO_X\to F_{*}W_n\sO_X).\]
Then $\mathcal{B}_{n,X}$ has a natural $\sO_X$-module structure via $R_{n-1}\colon W_n\sO_X\to \sO_X$ \cite[Proposition 3.6 (3)]{KTTWYY1}.
We set $\mathcal{B}_{X}\coloneqq \mathcal{B}_{1,X}$.
Then the exact sequence \eqref{eq:WWO} induces the short exact sequence
\begin{equation}\label{eq:BBB}
    0\to F_{*}\mathcal{B}_{n-1,X} \to \mathcal{B}_{n,X} \xrightarrow{R_{n-1}} \mathcal{B}_{X} \to 0
\end{equation}
of $\sO_X$-modules.
Consider the following pushout diagram :
\begin{center}
\begin{tikzcd}
0 \arrow[r] & W_n\sO_X \arrow{r}{F} \arrow[d,"R_{n-1}"'] & F_* W_n\sO_X \arrow{d} \arrow[r] & \mathcal{B}_{n,X}\arrow[d,equal]\arrow[r] & 0 \\
0 \arrow[r] & \sO_X \arrow{r}{\Phi_{X, n}} & Q_{X,n} \arrow[lu, phantom, "\usebox\pushoutdr" , very near start, yshift=0em, xshift=0.6em, color=black] \arrow[r]& \mathcal{B}_{n,X}\arrow[r] & 0
\end{tikzcd}
\end{center}
of $W_n\sO_X$-modules.
Then $Q_{X,n}$ has a natural $\sO_X$-module structure, and the lower short exact sequence can be seen as a short exact sequence of $\sO_X$-module \cite[Proposition 2.9 (2)]{KTTWYY1}.
Moreover, we obtain the following commutative diagram:
\begin{equation}\label{diagram:Q Vs O}
    \begin{tikzcd}
0 \arrow[r] & \sO_X \arrow{r}{\Phi_{X, n}}\arrow[d,equal] & Q_{X,n} \arrow{r}\arrow[d] & \mathcal{B}_{n,X}\arrow[r]\arrow[d,"R_{n-1}"] & 0\\
0 \arrow[r] & \sO_X \arrow{r}{F} & F_{*}\sO_X \arrow{r} & \mathcal{B}_{X}\arrow[r] & 0.
\end{tikzcd}
\end{equation}

\begin{defn}\label{def:quasi-F-injective}
    Let $X$ be a normal $F$-finite Noetherian scheme.
    We say that $X$ is \textit{quasi-$F$-injective at a closed point $x\in X$} if there exists a positive integer $n\in\Z_{>0}$ such that
    the map
    \[
    H^i_{\m_x}(\sO_{X})\to H^i_{\m_x}(Q_{X,n})
    \]
    is injective for all $i\geq 0$.
    We say that $X$ is \textit{quasi-$F$-injective} if it is quasi-$F$-injective at every closed point.
\end{defn}
\begin{rem}\label{rem:quasi-F-inj}
    Let $X$ be a normal $F$-finite Noetherian scheme.
    Recall that a quasi-$F$-purity at a closed point $x\in X$ is equivalent to the splitting of the map
    \[
    \sO_X \to Q_{X,n}
    \]
    for some $n>0$ at $x\in X$ (see \cite[Proposition 2.8]{KTTWYY1}).
    Thus, if $X$ is quasi-$F$-pure, then it is quasi-$F$-injective.
\end{rem}

\section{The inverse iterated Cartier operators}\label{subseq:Cartier}
Let $Y$ be a smooth variety, and let $E$ be a reduced divisor on $Y$ with snc support.
Let $A$ be a $\Q$-divisor whose support of the fractional part $A-\lfloor A\rfloor$ is contained in $E$.
We define $\sO_Y$-modules $B\Omega_Y^i(\log E)(pA)$ and $Z\Omega_Y^i(\log E)(pA)$ to be the boundaries and the cycles of $F_*\Omega_Y^i(\log E)(pA)$ at $i$: 
\begin{align} 
\label{eq:definition-of-B1} B \Omega_Y^i(\log E)(pA) &\coloneqq {\rm im}\left(F_*\Omega_Y^{i-1}(\log E)(pA) \xrightarrow{F_*d}  F_*\Omega_Y^{i}(\log E)(pA)\right),\\ 
Z\Omega_Y^i(\log E)(pA) &\coloneqq \Ker\left(F_*\Omega_Y^{i}(\log E)(pA) \xrightarrow{F_*d}  F_*\Omega_Y^{i+1}(\log E)(pA)\right) \nonumber,
\end{align}
where $F_*\Omega_Y^{i}(\log E)(pA)$ denotes $F_*\left(\Omega_Y^{i}(\log E)\otimes \sO_Y(\lfloor pA \rfloor)\right)$.
Note that, in general
\begin{align*}
    &B \Omega_Y^i(\log E)(pA)\neq B \Omega_Y^i(\log E)\otimes\sO_Y(pA)\quad\text{and}\\
    &Z\Omega_Y^i(\log E)(pA)\neq Z\Omega_Y^i(\log E)\otimes\sO_Y(pA).
\end{align*}
By definition, we obtain the short exact sequence
\begin{equation}\label{eq:ZOmegaB}
    0\to Z\Omega_Y^i(\log E)(pA) \to F_{*}\Omega_Y^i(\log E)(pA) \to B\Omega_Y^{i+1}(\log E)(pA) \to 0.
\end{equation}
Taking $i=0$ and $A=0$ in \eqref{eq:ZOmegaB}, we obtain the short exact sequence
\begin{equation}\label{eq:OOB}
0\to \sO_Y \to F_{*}\sO_Y \to B\Omega^{1}_Y \to 0.
\end{equation}
The Cartier isomorphism (\cite[Lemma 3.3]{Hara98}, \cite[equation (5.4.2)]{KTTWYY1}) gives the short exact sequence
\begin{equation} \label{eq:BZOmega}
0 \to B\Omega^i_X(\log E)(pA) \to Z\Omega^i_X(\log E)(pA) \xrightarrow{C} \Omega^i_X(\log E)(A) \to 0.
\end{equation}
Taking $i=d$ and $A=E=0$ in \eqref{eq:BZOmega}, we obtain the short exact sequence
\begin{equation}\label{eq:Bomegaomega}
0\to B\Omega^d_Y \to F_{*}\omega_Y \xrightarrow{C} \omega_Y \to 0,
\end{equation}
and the last map $C\colon F_{*}\omega_Y\to \omega_Y$ is nothing but the Frobenius trace map.

We set
\[
G\Omega^i_X(\log E)(pA) \coloneqq \frac{F_*\Omega^{i}_X(\log E)(pA)}{B\Omega^{i}_X(\log E)(pA)}.
\]
Then we have the inverse Cartier operator
\begin{align}
\label{eq:dual-Cartier-Hara} C^{-1}\colon \Omega^i_X(\log E)(A)&\underset{\cong}{\xleftarrow{C}} \frac{Z\Omega^{i}_X(\log E)(pA)}{B\Omega^{i}_X(\log E)(pA)} \\[0.2em]
&\xhookrightarrow{\hphantom{C}} \frac{F_*\Omega^{i}_X(\log E)(pA)}{B\Omega^{i}_X(\log E)(pA)}=G\Omega^i_X(\log E)(pA). \nonumber
 \end{align}

\subsubsection{Iterations}
Inductively on $n$, we will construct an $\sO_Y$-submodule \[
Z_n\Omega_Y^i(\log E)(p^nA) \subset F^n_* \Omega_Y(\log E)(p^nA),
\] 
and an $\sO_Y$-module homomorphism \[C\colon Z_{n}\Omega_Y^i(\log E)(p^nA) \onto Z_{n-1}\Omega_Y^i(\log E)(p^{n-1}A)\] as follows:
We first set 
\begin{align*}
    &Z_0\Omega_Y^i(\log E)(A)\coloneqq \Omega_Y^i(\log E)(A)\quad\text{and}\\
    &Z_1 \Omega_Y^i(\log E)(pA)\coloneqq Z\Omega_Y^i(\log E)(pA).
\end{align*}
The morphism $C\colon Z_1\Omega_Y^i(\log E)(pA) \onto Z_0\Omega_Y^i(\log E)(A)$ is induced by the Cartier isomorphism \eqref{eq:BZOmega}.
After the map \[C\colon Z_{n}\Omega_Y^i(\log E)(p^nA) \onto Z_{n-1}\Omega_Y^i(\log E)(p^{n-1}A)\] is constructed, we define an $\sO_Y$-submodule \[Z_{n+1}\Omega_Y^i(\log E)(p^{n+1}A)\subset F_*Z_{n}\Omega_Y^i(\log E)(p^{n}A)\] and an $\sO_Y$-moudle homomorphism \[C\colon Z_{n+1}\Omega_Y^i(\log E)(p^{n+1}A) \onto Z_n\Omega_Y^i(\log E)(p^nA)\] by the following pullback diagram:
\[
\begin{tikzcd}
Z_{n+1}\Omega_Y^i(\log E)(p^{n+1}A) \arrow[rd, phantom, "\usebox\pullback" , very near start, yshift=-0.3em, xshift=-0.6em, color=black] \arrow[d,"C"', twoheadrightarrow] \arrow[r, hook] & F_*Z_{n}\Omega_Y^i(\log E)(p^{n}A) \arrow[d, "F_*C"', twoheadrightarrow] \\
 Z_{n}\Omega_Y^i(\log E)(p^{n}A) \arrow[r, hook] & F_*Z_{n-1}\Omega_Y^i(\log E)(p^{n-1}A).
\end{tikzcd}    
\]
We denote by $C_{n}\colon Z_n\Omega_Y^i(\log E)(p^nA) \onto \Omega_Y^i(\log E)(A)$ the composite map
\[
Z_n\Omega_Y^i(\log E)(p^nA) \xrightarrow{C} Z_{n-1}\Omega_Y^i(\log E)(p^{n-1}A) \xrightarrow{C}\cdots \xrightarrow{C} \Omega_Y^i(\log E)(A).
\]
Here, we set \[C_0\coloneqq \mathrm{id}\colon Z_0\Omega_Y^i(\log E)(A)\to \Omega_Y^i(\log E)(A).\]
By induction on $n$ starting from \eqref{eq:ZOmegaB}, we have the following exact sequences
\begin{tiny}
\begin{equation}\label{diagram for Z_n}
\begin{tikzcd}
0 \arrow{r} & Z_{n+1}\Omega_Y^i(\log E)(p^nA) \arrow[rd, phantom, "\usebox\pullback" , very near start, yshift=-0.3em, xshift=-0.6em, color=black] \arrow[d,"C"', twoheadrightarrow] \arrow[r] & F_*Z_{n}\Omega_Y^i(\log E)(p^nA) \arrow[d, "F_{*}C"', twoheadrightarrow] \arrow[rr,"F_{*}d\circ F_{*}C_n"] && B\Omega^{i+1}_Y(\log E)(A) \arrow[d,equal] \arrow[r] & 0\\
0 \arrow{r} & Z_{n}\Omega_Y^i(\log E)(p^nA) \arrow{r} & F_*Z_{n-1}\Omega_Y^i(\log E)(p^{n-1}A) \arrow[rr,"F_*d\circ F_{*}C_{n-1}"] && B\Omega^{i+1}_Y(\log E)(A) \arrow{r} & 0,
\end{tikzcd}    
\end{equation}
\end{tiny}
\noindent where horizontal sequences are exact. 
Applying the snake lemma, we obtain the short exact sequence
\begin{equation}\label{eq:Z_ntoZ_{n-1}}
    0\to F^{n}_{*}B\Omega^i_Y(\log E)(p^{n+1}A) \to Z_{n+1}\Omega^i_Y(\log E)(p^{n+1}A) \xrightarrow{C} Z_{n}\Omega^i_Y(\log E)(p^{n}A) \to 0.
\end{equation}

We next define an $\sO_Y$-submodule \[
B_n \Omega_Y^i(\log E)(p^nA)\subset Z_n\Omega_Y^i(\log E)(p^nA)
\]
an $\sO_Y$-module homomorphism $C\colon B_n \Omega_Y^i(\log E)(p^nA) \onto B_{n-1}\Omega_Y^i(\log E)(p^{n-1}A)$.
We first set 
\begin{align*}
    &B_0\Omega_Y^i(\log E)(A)\coloneqq 0\quad\text{and}\\
    &B_1 \Omega_Y^i(\log E)(pA)\coloneqq B\Omega_Y^i(\log E)(pA).
\end{align*}
For $n\geq 2$, we define an $\sO_Y$-module $B_n \Omega_Y^i(\log E)(p^nA)$ so that it fits in the following pull-back diagram:
\begin{small}
\begin{equation}\label{diagram:construction for B_n}
\begin{tikzcd}
0 \arrow{r} & B_{n+1}\Omega_Y^i(\log E)(p^nA) \arrow[rd, phantom, "\usebox\pullback" , very near start, yshift=-0.3em, xshift=-0.6em, color=black] \arrow[d,"C"', twoheadrightarrow] \arrow{r} & Z_{n+1}\Omega_Y^i(\log E)(p^nA) \arrow[d,"C"', twoheadrightarrow] \arrow[r,"C_{n+1}"] & \Omega_Y^i(\log E)(A) \arrow[d,equal] \arrow[r] & 0\\
0 \arrow{r} & B_{n}\Omega_Y^i(\log E)(p^nA) \arrow{r} & Z_{n}\Omega_Y^i(\log E)(p^nA) \arrow[r,"C_{n}"] & \Omega_Y^i(\log E)(A) \arrow{r} & 0,
\end{tikzcd}    
\end{equation}
\end{small}
where horizontal sequences are exact. 
In particular, we obtain a short exact sequence
\begin{equation}\label{BZOmega(iterated)}
    0 \to B_{n}\Omega_Y^i(\log E)(p^nA) \to Z_{n}\Omega_Y^i(\log E)(p^nA) \xrightarrow{C_{n}}  \Omega_Y^i(\log E)(A) \to 0
    \end{equation}
for all $n\geq 0$.
When $n=1$, this is nothing but \eqref{eq:BZOmega}. 
By the snake lemma and \eqref{eq:Z_ntoZ_{n-1}}, we also obtain the short exact sequences
\begin{align}
    & 0\to F^{n}_{*}B\Omega_Y^i(\log E)(p^{n+1}A) \to B_{n+1}\Omega_Y^i(\log E)(p^{n+1}A) \xrightarrow{C} B_{n}\Omega_Y^i(\log E)(p^nA) \to 0,\label{eq:B_ntoB_{n-1}}\\
    & 0\to F_{*}B_{n}\Omega_Y^i(\log E)(p^{n+1}A) \to B_{n+1}\Omega_Y^i(\log E)(p^{n+1}A) \xrightarrow{C_{n}} B\Omega_Y^i(\log E)(pA) \to 0.\label{eq:B_ntoB_{n-1}2}
\end{align}
Recall that $Z_n\Omega^i_Y(\log E)(p^nA)$ and $B_n\Omega^i_Y(\log E)(p^nA)$ are locally free \cite[Lemma 5.10]{KTTWYY1}.

For $n\geq 0$, we set
\begin{equation}\label{def of GOmega}
    G_n\Omega^i_Y(\log E)(p^nA) \coloneqq \frac{F^n_*\Omega^{i}_Y(\log E)(p^nA)}{B_n\Omega^{i}_Y(\log E)(p^nA)}.
\end{equation}
Note that $G_0\Omega^i_Y(\log E)(A)=\Omega^{i}_Y(\log E)(A)$.
We have the inverse iterated Cartier operator
\begin{align}
\label{eq:dual-Cartier-Hara-higher} 
C^{-1}_{n}\colon \Omega^i_Y(\log E)(A)&\underset{\cong}{\xleftarrow{C_{n}}} \frac{Z_n\Omega^{i}_Y(\log E)(p^nA)}{B_n\Omega^{i}_Y(\log E)(p^nA)} \\[0.2em]
&\hookrightarrow \frac{F^n_*\Omega^{i}_Y(\log E)(p^nA)}{B_n\Omega^{i}_Y(\log E)(p^nA)}=G_n\Omega^i_Y(\log E)(p^nA). \nonumber
\end{align}

We can construct the inverse Cartier operator between $G_n\Omega^i_Y(\log E)(p^nA)$ and $G_{n+1}\Omega^i_Y(\log E)(p^{n+1}A)$ as follows:
\begin{align}\label{eq:GnGn+1}
  C^{-1}_{n+1,n}\colon   G_n\Omega^i_Y(\log E)(p^nA)= \frac{F^n_{*}\Omega^i_Y(\log E)(p^nA)}{B_n\Omega^i_Y(\log E)(p^nA)}&\xrightarrow{C^{-1}} \frac{F^n_{*}\left(\frac{F_{*}\Omega^i_Y(\log E)(p^{n+1}A)}{B\Omega^i_Y(\log E)(p^{n+1}A)}\right)}{\frac{B_{n+1}\Omega^i_Y(\log E)(p^{n+1}A)}{F^n_{*}B\Omega^i_Y(\log E)(p^{n+1}A)}}\\
  &\cong \frac{F^{n+1}_{*}\Omega^i_Y(\log E)(p^{n+1}A)}{B_{n+1}\Omega^i_Y(\log E)(p^{n+1}A)}\\
  &=G_{n+1}\Omega^i_Y(\log E)(p^{n+1}A),
\end{align}
where we use \eqref{eq:dual-Cartier-Hara} and \eqref{eq:B_ntoB_{n-1}} for the first map.
Then we have a decomposition 
\begin{equation}\label{eq:decomposition}
    C^{-1}_{n}\colon \Omega^i_Y(\log E)(A)\xrightarrow{C^{-1}=C_{1,0}^{-1}} G_1\Omega^i_Y(\log E)(pA)\xrightarrow{C_{2,1}^{-1}}\cdots \xrightarrow{C_{n,n-1}^{-1}} G_n\Omega^i_Y(\log E)(p^nA)
\end{equation}
of \eqref{eq:dual-Cartier-Hara-higher}.
Moreover, since we have
\begin{align*}
   G_n\Omega^i_Y(\log E)(p^nA)= \frac{F^n_{*}\Omega^i_Y(\log E)(p^nA)}{B_n\Omega^i_Y(\log E)(p^nA)}&\underset{\cong}{\xrightarrow{C^{-1}}} \frac{F^n_{*}\left(\frac{Z\Omega^i_Y(\log E)(p^{n+1}A)}{B\Omega^i_Y(\log E)(p^{n+1}A)}\right)}{\frac{B_{n+1}\Omega^i_Y(\log E)(p^{n+1}A)}{F^n_{*}B\Omega^i_Y(\log E)(p^{n+1}A)}}\\
   &\cong \frac{F^{n}_{*}Z\Omega^i_Y(\log E)(p^{n+1}A)}{B_{n+1}\Omega^i_Y(\log E)(p^{n+1}A)}, 
\end{align*}
we obtain a short exact sequence
\begin{small}
\begin{equation}\label{GGOmega/Z}
    0\to G_n\Omega^i_Y(\log E)(p^nA)\xrightarrow{C^{-1}_{n+1,n}}G_{n+1}\Omega^i_Y(\log E)(p^{n+1}A)\to F^n_{*}\left(\frac{F_{*}\Omega^i_Y(\log E)(p^{n+1}A)}{Z\Omega^i_Y(\log E)(p^{n+1}A)}\right)\to 0.
\end{equation}
\end{small}
By \eqref{eq:ZOmegaB}, the above short exact sequence can be written as follows:
\begin{small}
\begin{equation}\label{GGB}
    0\to G_n\Omega^i_Y(\log E)(p^nA)\xrightarrow{C^{-1}_{n+1,n}}G_{n+1}\Omega^i_Y(\log E)(p^{n+1}A)\to F^n_{*}B\Omega^{i+1}_Y(\log E)(p^{n+1}A)\to 0.
\end{equation}
\end{small}
In particular, since $B\Omega^{i+1}_Y(\log E)(p^{n+1}A)$ and $G_0\Omega^i_Y(\log E)(A)=\Omega^i_Y(\log E)(A)$ are locally free, we can observe that $G_n\Omega^i_Y(\log E)(p^nA)$ is locally free for all $n\geq 0$ inductively by \eqref{GGB}.
As a more general version of \eqref{GGOmega/Z}, we have a short exact sequence
\begin{small}
\begin{equation}\label{GGOmega/Z,iterated}
    0\to G_n\Omega^i_Y(\log E)(p^nA)\xrightarrow{C^{-1}_{n+l,n}}G_{n+l}\Omega^i_Y(\log E)(p^{n+l}A)\to F^n_{*}\left(\frac{F^l_{*}\Omega^i_Y(\log E)(p^{n+l}A)}{Z_l\Omega^i_Y(\log E)(p^{n+l}A)}\right)\to 0.
\end{equation}
\end{small}

For a $\Q$-divisor $A'$ such that $A\leq A'$ and $\Supp(A'-\lfloor A'\rfloor)\subset E$, we have a commutative diagram
\begin{equation}\label{digram:Omega and G}
    \begin{tikzcd}
\Omega_Y^i(\log E)(A) \arrow[d] \arrow{r} & G_{n}\Omega_Y^i(\log E)(p^nA) \arrow[d] \\
\Omega_Y^i(\log E)(A') \arrow{r} & G_{n}\Omega_Y^i(\log E)(p^nA'),
\end{tikzcd}  
\end{equation}
see \cite[Remark 2.7]{Kawakami-Witaszek}.

Finally, by \cite[Lemma 6.7]{KTTWYY2}, we have the commutative diagram
\begin{equation}\label{eq:WOWOB}
    \begin{tikzcd}
0 \arrow{r} & W_n\sO_Y \arrow[r,"F"]\arrow[d,"R_{n-1}"] & F_{*}W_n\sO_Y \arrow[r]\arrow[d,"F_*R_{n-1}"] & B_n\Omega^{1}_Y \arrow{r}\arrow[d,"C_{n-1}"] & 0\\
0 \arrow{r} &\sO_Y \arrow[r,"F"] & F_{*}\sO_Y \arrow[r,"d"] & B\Omega^{1}_Y  \arrow[r] & 0,
\end{tikzcd} 
\end{equation}
where horizontal sequences are exact. 

\begin{rem}\label{rem:B and BOmega}
Recall that we defined an $\sO_Y$-module $\mathcal{B}_{n,Y}$ by
\[
\mathrm{coker}(F\colon W_n\sO_Y\to F_{*}W_n\sO_Y).
\]
By \eqref{eq:WOWOB}, we can observe that the restriction map
$\mathcal{B}_{n,Y}\xrightarrow{R_{n-1}} \mathcal{B}_Y$
coincides with the Cartier operator
$B_n\Omega^{1}_Y\xrightarrow{C_{n-1}} B\Omega^{1}_Y$.
\end{rem}

\subsection{Special cases}

Taking $A=-\frac{1}{p^n}E$ in \eqref{BZOmega(iterated)} and \eqref{eq:dual-Cartier-Hara-higher}, we have a short exact sequence
\begin{equation}\label{BZOmega(iterated),-E}
   0 \to B_{n}\Omega_Y^i(\log E)(-E) \to Z_{n}\Omega_Y^i(\log E)(-E) \xrightarrow{C_{n}}  \Omega_Y^i(\log E)(-E) \to 0 
\end{equation}
and the iterated inverse Cartier operator
\begin{align}
\label{eq:dual-Cartier-Hara-higher,-E} 
C^{-1}_{n}\colon \Omega^i_Y(\log E)(-E)&\underset{\cong}{\xleftarrow{C_{n}}} \frac{Z_n\Omega^{i}_Y(\log E)(-E)}{B_n\Omega^{i}_Y(\log E)(-E)} \\[0.2em]
&\xhookrightarrow{\hphantom{C_{n,\Delta}}} \frac{F^n_*\Omega^{i}_Y(\log E)(-E)}{B_n\Omega^{i}_Y(\log E)(-E)}=G_n\Omega^i_Y(\log E)(-E). \nonumber
\end{align}
Taking $A=-\frac{1}{p^{n+1}}E$ in \eqref{eq:GnGn+1}, we have the inverse Cartier operator
\begin{align}\label{eq:GnGn+1,specail}
  C^{-1}_{n+1,n}\colon   G_n\Omega^i_Y(\log E)(-E)\to G_{n+1}\Omega^i_Y(\log E)(-E),
\end{align}
and a decomposition
\begin{equation}\label{eq:decomposition,specail}
    C^{-1}_{n}\colon \Omega^i_Y(\log E)(-E)\xrightarrow{C^{-1}=C_{1,0}^{-1}} G_1\Omega^i_Y(\log E)(-E)\xrightarrow{C_{2,1}^{-1}}\cdots \xrightarrow{C_{n,n-1}^{-1}} G_n\Omega^i_Y(\log E)(-E).
\end{equation}
Moreover, taking $A=-\frac{1}{p^{n+1}}E$ in \eqref{GGB}, we have a short exact sequence
\begin{equation}\label{GGB,specail}
    0\to G_n\Omega^i_Y(\log E)(-E)\xrightarrow{C^{-1}_{n+1,n}}G_{n+1}\Omega^i_Y(\log E)(-E)\to F^n_{*}B\Omega^{i+1}_Y(\log E)(-E)\to 0,
\end{equation}
and taking $A=-\frac{1}{p^{n+l}}E$ in \eqref{GGOmega/Z,iterated} we have a short exact sequence
\begin{small}
\begin{equation}\label{GGOmega/Z,iterated,specail}
    0\to G_n\Omega^i_Y(\log E)(-E)\xrightarrow{C^{-1}_{n+1,n}}G_{n+l}\Omega^i_Y(\log E)(-E)\to F^n_{*}\left(\frac{F^l_{*}\Omega^i_Y(\log E)(-E)}{Z_l\Omega^i_Y(\log E)(-E)}\right)\to 0.
\end{equation}
\end{small}

\begin{lem} \label{lem:preliminaries-duality} With notation as above, 
we have isomorphisms
\begin{enumerate}
\item[\textup{(1)}] $\sHom_{\sO_Y}(F_*\Omega^i_Y(\log E), \omega_Y) \cong F_*\Omega^{n-i}_Y(\log E)(-E)$,\\[-0.8em]
\item[\textup{(2)}] 
    $\sHom_{\sO_Y}(Z_n\Omega^i_Y(\log E), \omega_Y) \cong \frac{F^n_{*}\Omega^{i}_Y(\log E)(-E)}{B_n\Omega^{i}_Y(\log E)(-E)}= G_n\Omega^{n-i}_Y(\log E)(-E)$, and\\[-0.8em] 
    \item[\textup{(3)}] $\sHom_{\sO_Y}(B_n\Omega^{i}_Y(\log E), \omega_Y) \cong \frac{F^n_{*}\Omega^{i}_Y(\log E)(-E)}{Z_n\Omega^{i}_Y(\log E)(-E)}$, 
\end{enumerate}
\end{lem}
\begin{proof}
    (1) follows from the usual duality (see \cite[Lemma 2.8]{Kawakami-Witaszek} for example). 
    Moreover, by \cite[Lemma 2.8]{Kawakami-Witaszek}, an isomorphism 
    \[
    \sHom_{\sO_Y}(B\Omega^i_Y(\log E), \omega_Y) \cong B\Omega^{d-i+1}(\log E)(-E).
    \]
    We first prove (2) by induction on $n$.
    Suppose that 
    \[
    \sHom_{\sO_Y}(Z_{n-1}\Omega^i_Y(\log E),\omega_Y)\cong \frac{F^{n-1}_{*}\Omega^{i}_Y(\log E)(-E)}{B_{n-1}\Omega^{i}_Y(\log E)(-E)}.
    \]
    Then, taking $\sHom_{\sO_X}(-,\omega_Y)$ of the short exact sequence (cf.~\eqref{eq:Z_ntoZ_{n-1}})
    \[
    0\to Z_n\Omega^i_Y(\log E)\to F_{*}Z_{n-1}\Omega^i_Y(\log E) \to B\Omega^{i+1}_Y(\log E)\to 0,
    \]
    we have a short exact sequence
    \begin{multline*}
        0\to B\Omega^{d-i}_Y(\log E)(-E) \xrightarrow{C^{-1}_n} F_{*}\left(\frac{F^{n-1}_{*}\Omega^{d-i}_Y(\log E)(-E)}{B_{n-1}\Omega^{d-1}_Y(\log E)(-E)}\right)\\\to \sHom_{\sO_Y}(Z_n\Omega^i_Y(\log E), \omega_Y) \to 0.
    \end{multline*}
    Since we have an isomorphism (see \eqref{eq:B_ntoB_{n-1}2})
    \[
    B\Omega^{d-i}_Y(\log E)(-E) \underset{\cong}{\xrightarrow{C^{-1}_n}} \frac{B_n\Omega^{d-i}_Y(\log E)(-E)}{F_{*}B_{n-1}\Omega^{d-1}_Y(\log E)(-E)},
    \]
    we have
    \[
    \sHom_{\sO_Y}(Z_n\Omega^i_Y(\log E), \omega_Y)\cong \frac{F^n_{*}\Omega^{i}_Y(\log E)(-E)}{B_n\Omega^{i}_Y(\log E)(-E)}=G_n\Omega^{n-i}_Y(\log E)(-E).
    \]

    Finally, we prove (3).
    Taking $\sHom_{\sO_X}(-,\omega_Y)$ of the short exact sequence
    \[
    0\to B_n\Omega^i_Y(\log E)\to Z_{n}\Omega^i_Y(\log E) \xrightarrow{C_n} \Omega^{i}_Y(\log E)\to 0,
    \]
    we have
    \[
    0\to \Omega^{d-i}_Y(\log E)(-E) \xrightarrow{C^{-1}_n} G_n\Omega^{d-i}_Y(\log E)(-E) \to \sHom_{\sO_Y}(B_n\Omega^i_Y(\log E), \omega_Y) \to 0,
    \]
    where we used (2) for the second term.
    Since we have an isomorphism \eqref{BZOmega(iterated),-E}
    \[
    \Omega^{d-i}_Y(\log E)(-E) \underset{\cong}{\xrightarrow{C^{-1}_n}} \frac{Z_n\Omega^{d-i}_Y(\log E)(-E)}{B_{n}\Omega^{d-1}_Y(\log E)(-E)},
    \]
    we obtain an isomorphism
    \[
    \sHom_{\sO_Y}(B_n\Omega^i_Y(\log E), \omega_Y)\cong \frac{F^n_{*}\Omega^{d-i}_Y(\log E)(-E)}{Z_n\Omega^{d-i}_Y(\log E)(-E)}.
    \]
    Thus, we conclude.
\end{proof}

\begin{lem}\label{lem:FOmega^{d-1}/ZOmega}
With notation as above, set $d=\dim Y$.
Then we have the following isomorphisms:
\begin{enumerate}
    \item[\textup{(1)}] $B_n\Omega^1_Y=B_n\Omega^1_Y(\log E)$ for all $n\geq 1$.
    \item[\textup{(2)}] $\frac{F^n_{*}\Omega^{d-1}_Y}{Z_n\Omega^{d-1}_Y}=\frac{F^n_{*}\Omega^{d-1}_Y(\log E)(-E)}{Z_n\Omega^{d-1}_Y(\log E)(-E)}$ for all $n\geq 1$.
    \item[\textup{(3)}] $B\Omega^d_Y\cong B\Omega^d_Y(\log E)(-E)$.
\end{enumerate}
\end{lem}
\begin{proof}
By definition, we have
    \begin{align} 
B \Omega_Y^1 &\coloneqq \mathrm{im}\left(F_*\sO_Y \xrightarrow{F_*d}  F_*\Omega_Y^{1}\right),\\ 
B\Omega_Y^1(\log E)&\coloneqq \mathrm{im}\left(F_*\sO_X \xrightarrow{F_*d}  F_*\Omega_Y^{1}(\log E)\right) \nonumber,
\end{align}
and thus (1) holds by \eqref{eq:B_ntoB_{n-1}}. (2) follows from (1) and duality (Lemma \ref{lem:preliminaries-duality}).
(3) is a consequence of the case $n=1$ of (2) and \eqref{eq:ZOmegaB}.
\end{proof}

\section{Proof of main theorems}

In this section, we prove the main theorems.

\begin{thm}[=Theorem \ref{Intromain:up to C}]\label{thm:steenbring vanishing up to nilpotent}
    Let $X$ be a normal variety with $d\coloneqq \dim\,X\geq 2$, and
    let $\pi\colon Y\to X$ be a log resolution such that the reduced $\pi$-exceptional divisor $E$ supports a $\pi$-ample divisor.
    Then 
    \[
    C^{-1}_{n}\colon R^{d-1}\pi_{*}\Omega^i_Y(\log E)(-E)\to R^{d-1}\pi_{*}G_n\Omega^i_Y(\log E)(-E)
    \]
    is a zero map for all $i\geq 0$ and $n\gg0$.
\end{thm}
\begin{proof}
   Fix $i\geq 0$.
   By assumption, we can take a $\pi$-ample anti-effective $\Q$-divisor $A$ such that $\lceil A \rceil=0$ and $\Supp(A)=E$.
   Take $n\gg0$ so that $A\leq -\frac{1}{p^n}E$ and $R^{d-1}\pi_{*}\Omega^{i}_Y(\log E)(p^nA)=0$ whose existence is ensured by Serre vanishing.
   By taking $A'=-\frac{1}{p^n}E$ in the diagram \eqref{digram:Omega and G}, we have the following commutative diagram:
   \[
\begin{tikzcd}
\Omega_Y^i(\log E)(A) \arrow[d,equal] \arrow[r,"C^{-1}_n"] & G_{n}\Omega_Y^i(\log E)(p^nA) \arrow[d] \\
\Omega_Y^i(\log E)(-E) \arrow[r,"C^{-1}_n"] & G_{n}\Omega_Y^i(\log E)(-E).
\end{tikzcd}  
\]
Here, we used equalities $\lfloor A'\rfloor= \lfloor p^nA'\rfloor =-E$.
Thus, it suffices to show that
   \[
   R^{d-1}\pi_{*}G_n\Omega^{i}_Y(\log E)(p^nA)=0.
   \]
   Consider the short exact sequence (see \eqref{def of GOmega})
   \[
   0\to B_n\Omega^{i}_Y(\log E)(p^nA)\to F^n_{*}\Omega^{i}_Y(\log E)(p^nA)\to G_n\Omega^{i}_Y(\log E)(p^nA)\to 0.
   \]
   Then we have an exact sequence
   \begin{small}
     \[
   F^n_{*}R^{d-1}\pi_{*}\Omega^{i}_Y(\log E)(p^nA)\to R^{d-1}\pi_{*}G_n\Omega^{i}_Y(\log E)(p^nA) \to R^d\pi_{*}B_n\Omega^{i}_Y(\log E)(p^nA)=0.
   \]  
   \end{small}
   By the deinition of $n$, we have
   \[
   R^{d-1}\pi_{*}\Omega^{i}_Y(\log E)(p^nA)=0.
   \]
   Therefore, we obtain
   \[
   R^{d-1}\pi_{*}G_n\Omega^{i}_Y(\log E)(p^nA)=0,
   \]
   and the assertion holds.
\end{proof}

\begin{prop}\label{prop:higherdim}
    Let $X$ be a normal variety with $d\coloneqq \dim\,X$, and
    let $\pi\colon Y\to X$ be a log resolution with reduced divisor $E$. 
    If $R^{j-2}\pi_{*}\omega_Y=R^{j-1}\pi_{*}\omega_Y=0$, then 
    \[
    C^{-1}_{n}\colon R^{j}\pi_{*}\Omega^{d-1}_Y(\log E)(-E)\to R^{j}\pi_{*}G_n\Omega^{d-1}_Y(\log E)(-E)
    \]
    is injective for all $n\geq0$.
\end{prop}
\begin{proof}
    By \eqref{eq:Bomegaomega}, we have the short exact sequence 
   \[
    0\to B\Omega^d_Y \to F_{*}\omega_Y \xrightarrow{C} \omega_Y\to 0.
    \]
   Then, by assumption that $R^{j-2}\pi_{*}\omega_Y=R^{j-1}\pi_{*}\omega_Y=0$, we have $R^{j-1}\pi_{*}B\Omega^{d}_Y=0$.
    By \eqref{GGB,specail}, we have a short exact sequence
    \begin{small}
        \[
    0\to G_n\Omega^{d-1}_Y(\log E)(-E)\xrightarrow{C^{-1}_{n+1,n}} G_{n+1}\Omega^{d-1}_Y(\log E)(-E)\to F^n_{*}B\Omega^{d}_Y(\log E)(-E)\to 0
    \]
    \end{small}
    for all $n\geq 0$.
    Furthermore, by Lemma \ref{lem:FOmega^{d-1}/ZOmega} (3),
    we have an isomorphism
    \[
    B\Omega^{d}_Y(\log E)(-E)\cong B\Omega^d_Y.
    \]
    Thus, the map
    \[
    R^{j}\pi_{*}G_n\Omega^{d-1}(\log E)(-E)\xrightarrow{C_{n+1,n}^{-1}} R^{j}\pi_{*}G_{n+1}\Omega^{d-1}(\log E)(-E)
    \]
    is injective for all $n\geq0$.
    By \eqref{eq:decomposition,specail}, we have a decomposition
    \[
    C_n^{-1}=C_{n,n-1}^{-1} \circ\cdots \circ C_{1,0}^{-1} 
    \]
    and thus the map
    \[
    C_{n}^{-1}\colon R^{d-1}\pi_{*}\Omega^{d-1}_Y(\log E)(-E)\to R^{d-1}\pi_{*}G_n\Omega^{d-1}_Y(\log E)(-E)
    \]
    is injective for all $n\geq0$, as desired.
\end{proof}

\begin{lem}\label{lem:vanishing of higher B}
    Let $X$ be a normal variety, and let $\pi\colon Y\to X$ be a resolution.
    Suppose that $R\pi_{*}\sO_Y\cong \sO_X$. 
    Then $R\pi_{*}B_n\Omega^1_Y\cong \mathcal{B}_{n,X}$ for all $n\geq 1$.
\end{lem}
\begin{proof}
    By the short exact sequence \eqref{eq:WWO}
    \[
    0\to F_{*}W_{n-1}\sO_Y\xrightarrow{V} W_n\sO_Y\xrightarrow{R_{n-1}} \sO_Y\to 0,
    \]
    we have $R\pi_{*}W_n\sO_Y=W_n\sO_X$ for all $n\geq1$.
    By the short exact sequence \eqref{eq:WOWOB}
    \[
    0\to W_{n}\sO_Y\xrightarrow{F} F_{*}W_n\sO_Y\to B_n\Omega^1_Y\to 0,
    \]
    we conclude.
\end{proof}

\begin{defn}\label{def:rational singularities}
    Let $X$ be a normal variety.
    We say that $X$ has \textit{rational singularities} if $X$ is Cohen--Macaulay, and for any resolution $\pi\colon Y\to X$, an isomorphism $R\pi_{*}\sO_Y=\sO_X$ holds.
\end{defn}
\begin{rem}\label{rem:rational}
   If $X$ has rational singularities, then $R^j\pi_{*}\omega_Y=0$ for all $j>0$.
   In fact, from local duality we can observe that the vanishing $R^j\pi_{*}\omega_Y=0$ is equivalent to $H^{d-j}_{\m_x}(\sO_X)=0$ for all $x\in X$.
\end{rem}

\begin{prop}\label{prop:threedim}
    Let $X$ be a normal variety with rational singularities with $d\coloneqq \dim X$, and let $\pi\colon Y\to X$ be a log resolution with reduced exceptional divisor $E$.
    Suppose that $X$ is quasi-$F$-injective.
    Then 
    \[
    C^{-1}_{n}\colon R^{2}\pi_{*}\Omega^{d-1}_Y(\log E)(-E)\to R^{2}\pi_{*}G_n\Omega^{d-1}_Y(\log E)(-E)
    \]
    is injective for all $n\geq 0$.
\end{prop}
\begin{proof}
    Since the assertion is local on $X$, we may assume that $X=\Spec R$ is affine.
    If $d\leq 2$, then the assertion is obvious since $R^{2}\pi_{*}\Omega^{d-1}_Y(\log E)(-E)=0$.
    Thus, we may assume that $d\geq 3$.
    Let $x\in X$ be a closed point, and let $\m\subset R$ be the maximum ideal corresponds to $x\in X$.
    
    Since $X$ is quasi-$F$-injective, we can take $l\geq1$ such that $H^{d}_{\m}(\sO_X)\to H^{d}_{\m}(Q_{X,l})$ is injective.
By \eqref{diagram:Q Vs O}, we have the following commutative diagram
\[
\begin{tikzcd}
H^{d-1}_{\m}(Q_{X,l})\arrow[r, twoheadrightarrow]\arrow[d] & H^{d-1}_{\m}(\mathcal{B}_{l,X})\arrow[d,"R_{l-1}"]\\
  \mathllap{0\,=\,\, } H^{d-1}_{\m}(F_{*}\sO_X)\arrow[r]& H^{d-1}_{\m}(\mathcal{B}_{X}),
\end{tikzcd}
\]
where the top horizontal arrow is surjective $H^{d}_{\m}(\sO_X)\to H^{d}_{\m}(Q_{X,l})$ is injective, and $H^{d-1}_{\m}(F_{*}\sO_X)=0$ since $X$ is Cohen--Macaulay.
Thus, the restriction 
\[
H^{d-1}_{\m}(\mathcal{B}_{l,X})\xrightarrow{R_{l-1}} H^{d-1}_{\m}(\mathcal{B}_X)
\] is a zero map.
By Lemma \ref{lem:vanishing of higher B} and Remark \ref{rem:B and BOmega}, the above map coincides with a map
\[
H^{d-1}_{\m}(R\pi_{*}B_l\Omega^1_Y)\xrightarrow{C_{l-1}} H^{d-1}_{\m}(R\pi_{*}B\Omega^1_Y),
\]
and thus this is also a zero map.
Recall that $B_n\Omega^1_Y=B_n\Omega^1_Y(\log E)$ for all $n\geq 0$ by Lemma \ref{lem:FOmega^{d-1}/ZOmega} (1).
Thus, local duality \cite[\href{https://stacks.math.columbia.edu/tag/0AAK}{Tag 0AAK}]{stacks-project} yields that
\[
\mathcal{H}^{-d+1}R\sHom_{\sO_X}(R\pi_{*}B\Omega^1_Y(\log E),\omega^{\bullet}_X)\to \mathcal{H}^{-d+1}R\sHom_{\sO_X}(R\pi_{*}B_l\Omega^1_Y(\log E),\omega^{\bullet}_X)
\]
is a zero map.
By Grothendieck duality, 
\[
R^1\pi_{*}\sHom_{\sO_Y}(B\Omega^1_Y(\log E),\omega_Y)\to R^1\pi_{*}\sHom_{\sO_Y}(B_l\Omega^1_Y(\log E),\omega_Y)
\]
is a zero map.
Since we have
\[
\sHom_{\sO_Y}(B_n\Omega^1_Y(\log E),\omega_Y)\cong \frac{F^n_{*}\Omega^{d-1}_Y(\log E)(-E)}{Z_n\Omega^{d-1}_Y(\log E)(-E)}
\]
for all $n\geq 0$ by Lemma \ref{lem:preliminaries-duality} (3),
\[
R^{1}\pi_{*}\left(\frac{F_{*}\Omega^{d-1}_Y(\log E)(-E)}{Z\Omega^{d-1}_Y(\log E)(-E)}\right)\to R^{1}\pi_{*}\left(\frac{F^l_{*}\Omega^{d-1}_Y(\log E)(-E)}{Z_l\Omega^{d-1}_Y(\log E)(-E)}\right)
\]
is a zero map.
By \eqref{GGOmega/Z,iterated,specail}, we have the following commutative diagram:
\begin{tiny}
\[
\begin{tikzcd}
F^n_{*}R^{1}\pi_{*}\left(\frac{F_{*}\Omega^{d-1}_Y(\log E)(-E)}{Z\Omega^{d-1}_Y(\log E)(-E)}\right)\arrow[r]\arrow[d,"0"] & R^{2}\pi_{*}G_n\Omega^{d-1}_Y(\log E)(-E) \arrow[r,"C_{n+1,n}^{-1}"]\arrow[d,equal] & R^{2}\pi_{*}G_{n+1}\Omega^{d-1}_Y(\log E)(-E) \arrow[d,"C^{-1}_{n+l,n+1}"] \\
F^n_{*}R^{1}\pi_{*}\left(\frac{F^l_{*}\Omega^{d-1}_Y(\log E)(-E)}{Z_l\Omega^{d-1}_Y(\log E)(-E)}\right)\arrow[r] &R^{2}\pi_{*}G_n\Omega^{d-1}_Y(\log E)(-E)\arrow[r,"C^{-1}_{n+l,n}"]& R^{2}\pi_{*}G_{n+l}\Omega^{d-1}_Y(\log E)(-E)
\end{tikzcd}
\]
\end{tiny}
for all $n\geq 0$.
By diagram chasing, 
\[
C^{-1}_{n+1,n}\colon R^{2}\pi_{*}G_n\Omega^{d-1}_Y(\log E)(-E)\to R^{2}\pi_{*}G_{n+1}\Omega^{d-1}_Y(\log E)(-E)
\]
is injective for all $n\geq 0$.
Since we have a decomposition \eqref{eq:decomposition,specail}, 
    \[
    C_n^{-1}=C_{n,n-1}^{-1} \circ\cdots \circ C_{1,0}^{-1}, 
    \]
    the map
\[
    C_{n}^{-1}\colon R^{2}\pi_{*}\Omega^{d-1}_Y(\log E)(-E)\to R^{2}\pi_{*}G_n\Omega^{d-1}_Y(\log E)(-E)
    \]
    is injective for all $n\geq 0$, as desired.
\end{proof}

\begin{proof}[Proof of Theorem \ref{introthm:injective action}]
    The assertion follows from Propositions \ref{prop:higherdim} and \ref{prop:threedim}.
\end{proof}

\begin{proof}[Proof of Theorem \ref{Introthm:F-inj}]
    The assertion follows from Theorems \ref{Intromain:up to C} and \ref{introthm:injective action}.
\end{proof}

\begin{proof}[Proof of Theorem \ref{Introthm:SFR}]
    Since $X$ is strongly $F$-regular, it is $F$-injective.
    By \cite[Theorem 3.5]{Baudin-Kawakami-Rysler}, $X$ has rational singularities.
    Therefore, the assertion follows from Theorem \ref{Introthm:F-inj}.
\end{proof}

\begin{proof}[Proof of Theorem \ref{Introthm:klt}]
    By \cite[Corollary 1.3]{ABL}, $X$ has rational singularities.
    By \cite[Theorem A]{KTTWYY2}, $X$ is quasi-$F$-pure, and in particular, quasi-$F$-injective (Remark \ref{rem:quasi-F-inj}).
    Then the assertion follows from Theorems \ref{Intromain:up to C} and \ref{introthm:injective action}.
\end{proof}

\begin{eg}\label{eg:counterex}
   Fix a positive integer $d\geq 3$ and a prime number $p$.
   By \cite[Theorem 2 and Corollary on p.~519]{Mukai}, there exists a smooth projective variety $Z$ of $\dim Z=d-1$ and an ample Cartier divisor $A$ on $Z$ such that $H^0(Z,\Omega^1_Z(-A))\neq 0$.
   Let $X$ be the affine cone $\Spec \bigoplus_{m\geq 0} H^0(Z,\sO_Z(mA))$, and let $\pi\colon Y\to X$ be the blow-up at the vertex of the affine cone.
   Then $\pi\colon Y\to X$ is a log resolution, and 
   \[
   R^{d-1}\pi_{*}\Omega^{d-1}_Y(\log E)(-E)\neq 0 .
   \]
   In fact, by the reside exact sequence, we have a short exact sequence
   \[
   0\to \Omega^{d-1}_Y(-E) \to  \Omega^{d-1}_Y(\log E)(-E) \to \Omega^{d-2}_E(-E) \to 0.
   \]
   Then we have an exact sequence
   \[
   R^{d-1}\pi_{*}\Omega^{d-1}_Y(\log E)(-E) \to R^{d-1}\pi_{*}\Omega^{d-2}_E(-E) \to R^d\pi_{*}\Omega^{d-1}_Y(-E)=0.
   \]
   Since we have
   \[
   R^{d-1}\pi_{*}\Omega^{d-1}_E(-E)=H^{d-1}(Z, \Omega^{d-1}_Z(A))\cong H^0(Z,\Omega^1_Z(-A))\neq 0,
   \]
   we conclude.
\end{eg}

\begin{rem}
    It is known that Steenbrink vanishing is related to extension theorems for differential forms (see \cite{GKK}).  
Let $X$ be a normal variety, and let $\pi\colon Y \to X$ be a log resolution with reduced exceptional divisor $E$.  
We consider the \emph{logarithmic extension property} for differential forms on $X$, namely, the surjectivity of the natural restriction map:
\[
\pi_{*}\Omega^i_Y(\log E) \hookrightarrow \Omega^{[i]}_X,
\]
where $\Omega^{[i]}_X$ denotes the sheaf of reflexive differential $i$-forms on $X$.  
Since the property is local on $X$, we may assume $X$ is affine.  
In this setting, we have the exact sequence:
\[
0\to \pi_{*}\Omega^i_Y(\log E) \to \Omega^{[i]}_X \to H^{1}_E(Y,\Omega^i_Y(\log E)),
\]
so that the logarithmic extension property reduces to the vanishing:
\[
H^{1}_E(Y,\Omega^i_Y(\log E)) = 0.
\]
Assume further that $X$ has only isolated singularities.  
Then, by formal duality, we obtain:
\[
H^{1}_E(Y,\Omega^i_Y(\log E)) \cong R^{d-1}\pi_{*}\Omega^{d-i}_Y(\log E)(-E),
\]
which shows that the extension property follows from Steenbrink-type vanishing.  
In particular, Theorems \ref{Introthm:F-inj}, \ref{Introthm:SFR} and \ref{Introthm:klt} yield the logarithmic extension property for one-forms on isolated
rational singularities that has dimension $\geq 4$ or dimension $\geq 3$ and $F$-injective, 
on isolated 3-dimensional strongly $F$-regular singularities, and on isolated 3-dimensional klt threefolds in $p>41$, when $E$ supports a $\pi$-ample divisor.
However, stronger results have already been established in \cite[Theorem D]{Kaw4}, \cite[Theorem A]{Kawakami-Sato}, and \cite[Theorem B]{Kawakami-Sato} respectively.  
For example, it is shown in \cite[Theorem A]{Kawakami-Sato} that the logarithmic extension property holds in all dimensions for strongly $F$-regular varieties and all proper birational morphisms $f\colon Z\to X$, without any assumption on the codimension of the singular locus.
\end{rem}


\section*{Acknowledgements}
The author expresses his gratitude to Jefferson Baudin and Linus Rösler for valuable conversations.
He was supported by JSPS KAKENHI Grant number JP24K16897.


\begin{thebibliography}{GKKP11}

\bibitem[ABL22]{ABL}
Emelie Arvidsson, Fabio Bernasconi, and Justin Lacini.
\newblock On the {K}awamata-{V}iehweg vanishing theorem for log del {P}ezzo surfaces in positive characteristic.
\newblock {\em Compos. Math.}, 158(4):750--763, 2022.

\bibitem[BK23]{Bernasconi-Kollar}
Fabio Bernasconi and J\'anos Koll\'ar.
\newblock Vanishing theorems for three-folds in characteristic {$p>5$}.
\newblock {\em Int. Math. Res. Not. IMRN}, (4):2846--2866, 2023.

\bibitem[BKR25]{Baudin-Kawakami-Rysler}
Jefferson Baudin, Tatsuro Kawakami, and Linus R\"osler.
\newblock On {G}rauert–{R}iemenschneider vanishing for {C}ohen--{M}acaulay schemes of klt type.
\newblock arXiv:2506.21381, 2025.

\bibitem[CR11]{CR11}
Andre Chatzistamatiou and Kay R\"{u}lling.
\newblock Higher direct images of the structure sheaf in positive characteristic.
\newblock {\em Algebra Number Theory}, 5(6):693--775, 2011.

\bibitem[CR15]{CR15}
Andre Chatzistamatiou and Kay R\"ulling.
\newblock Vanishing of the higher direct images of the structure sheaf.
\newblock {\em Compos. Math.}, 151(11):2131--2144, 2015.

\bibitem[GKK10]{GKK}
Daniel Greb, Stefan Kebekus, and S\'{a}ndor~J. Kov\'{a}cs.
\newblock Extension theorems for differential forms and {B}ogomolov-{S}ommese vanishing on log canonical varieties.
\newblock {\em Compos. Math.}, 146(1):193--219, 2010.

\bibitem[GKKP11]{GKKP}
Daniel Greb, Stefan Kebekus, S\'{a}ndor~J. Kov\'{a}cs, and Thomas Peternell.
\newblock Differential forms on log canonical spaces.
\newblock {\em Publ. Math. Inst. Hautes \'{E}tudes Sci.}, 114:87--169, 2011.

\bibitem[Har98]{Hara98}
Nobuo Hara.
\newblock A characterization of rational singularities in terms of injectivity of {F}robenius maps.
\newblock {\em Amer. J. Math.}, 120(5):981--996, 1998.

\bibitem[Kaw22]{Kaw4}
Tatsuro Kawakami.
\newblock Extendability of differential forms via {C}artier operators.
\newblock \url{https://arxiv.org/abs/2207.13967v4}, 2022.
\newblock To appear in \emph{J.~Eur.~Math.~Soc.~(JEMS)}.

\bibitem[Kaw24]{Kaw6}
Tatsuro Kawakami.
\newblock Steenbrink-type vanishing for surfaces in positive characteristic.
\newblock {\em Bull. Lond. Math. Soc.}, 56(11):3484--3501, 2024.

\bibitem[Kov14]{Kovacs(Steenbrink)}
S\'andor~J. Kov\'acs.
\newblock Steenbrink vanishing extended.
\newblock {\em Bull. Braz. Math. Soc. (N.S.)}, 45(4):753--765, 2014.

\bibitem[KS25]{Kawakami-Sato}
Tatsuro Kawakami and Kenta Sato.
\newblock Extending one-forms on {$F$}-regular singularities.
\newblock \url{https://arxiv.org/abs/2502.17148}, 2025.

\bibitem[KTT{\etalchar{+}}22]{KTTWYY1}
Tatsuro Kawakami, Teppei Takamatsu, Hiromu Tanaka, Jakub Witaszek, Fuetaro Yobuko, and Shou Yoshikawa.
\newblock Quasi-{F}-splittings in birational geometry.
\newblock \url{https://arxiv.org/abs/2208.08016}, 2022.
\newblock To appear in \emph{Ann. Sci.~\'Ec.~Norm.~Sup\'er.~(4)}.

\bibitem[KTT{\etalchar{+}}24]{KTTWYY2}
Tatsuro Kawakami, Teppei Takamatsu, Hiromu Tanaka, Jakub Witaszek, Fuetaro Yobuko, and Shou Yoshikawa.
\newblock Quasi-{F}-splittings in birational geometry {II}.
\newblock {\em Proc. Lond. Math. Soc. (3)}, 128(4):Paper No. e12593, 81, 2024.

\bibitem[KW24a]{Kawakami-Witaszek}
Tatsuro Kawakami and Jakub Witaszek.
\newblock Higher {F}-injective singularities.
\newblock \url{https://arxiv.org/abs/2412.08887}, 2024.

\bibitem[KW24b]{Kollar-Witaszek}
J{\'a}nos Koll{\'a}r and Jakub Witaszek.
\newblock Resolution and alteration with ample exceptional divisor.
\newblock {\em Science China Mathematics}, pages 1--4, 2024.

\bibitem[Muk13]{Mukai}
Shigeru Mukai.
\newblock Counterexamples to {K}odaira's vanishing and {Y}au's inequality in positive characteristics.
\newblock {\em Kyoto J. Math.}, 53(2):515--532, 2013.

\bibitem[SPA25]{stacks-project}
The Stacks Project~Authors.
\newblock The {S}tacks {P}roject.
\newblock \url{https://stacks.math.columbia.edu/}, 2025.

\bibitem[Ste85]{Steenbrink}
J.~H.~M. Steenbrink.
\newblock Vanishing theorems on singular spaces.
\newblock Number 130, pages 330--341. 1985.
\newblock Differential systems and singularities (Luminy, 1983).

\end{thebibliography}
\newcommand{\etalchar}[1]{$^{#1}$}

\bigskip

\end{document}